\documentclass[a4paper,11pt,oneside]{amsart}

\usepackage{geometry}
\usepackage{graphicx, amsfonts, amssymb}
\usepackage{mathrsfs, amsmath, amsthm}
\usepackage{verbatim}
\usepackage{relsize}
\usepackage{enumerate}
\usepackage{hyperref}
\usepackage{version}
\usepackage{blindtext}
\usepackage{mathtools}
\usepackage{color}
\usepackage{ stmaryrd }
\usepackage{ esint }
\usepackage{bigints}
\usepackage{bbm}

\newtheorem{theorem}{Theorem}

\newtheorem{Lemm}[theorem]{Lemma}
\newtheorem{prop}[theorem]{Proposition}
\newtheorem{corollary}[theorem]{Corollary}

\newtheorem{foundationaltheorem}{\textbf{Theorem}}

\theoremstyle{remark}

\DeclareMathOperator{\BMOA}{BMOA}
\DeclareMathOperator{\VMOA}{VMOA}
\DeclareMathOperator{\BMO}{BMO}
\DeclareMathOperator{\VMO}{VMO}

\DeclareMathOperator{\QA}{QA}

\DeclareMathOperator{\A}{A (\mathbb{D})}

\hypersetup{
	colorlinks=true,
	linkcolor=blue,
	filecolor=magenta,      
	urlcolor=cyan,
	pdftitle={Meromorphic Optimal Domain of Integral Operators},
	pdfpagemode=FullScreen,
}

\title{Carleson Measures, Vanishing Mean Oscillation and Critical Points}

\author[C. Bellavita]{Carlo Bellavita}
\email{carlo.bellavita@gmail.com}
\email{carlobellavita@ub.edu}
\address{Departament de Matem\`atica i Inform\`atica, Universitat de Barcelona, Gran Via 585, 08007 Barcelona, Spain.}
\address{Current Address: Dipartimento di Matematica, Universit\'a degli Studi di Milano, Via C. Saldini 50, 20133, Milano, Italy}

\author[A. Nicolau]{Artur Nicolau}
\email{artur.nicolau@uab.cat}
\address{Departament de Matem\`atiques, Universitat Aut\`onoma de Barcelona and Centre de Recerca Matem\`atica, 08193 Barcelona, Spain.}

\author[G. Stylogiannis]{Georgios Stylogiannis}
\email{g.stylog@gmail.com}
\email{stylog@math.auth.gr}
\address{Department of Mathematics, Aristotle University of Thessaloniki, 54124 Thessaloniki, Greece.}



\begin{document}

\begin{abstract}
    Given a finite positive Borel measure $\mu$ in the open unit disc of the complex plane, we construct a bounded outer function $E$ whose boundary values have vanishing mean oscillation such that $|E| \mu$ is a vanishing Carleson measure. As an application it is shown that given any function in a Hardy space, there exists a bounded analytic function in the unit disc whose boundary values have vanishing mean oscillation, with the same critical points and multiplicities. 
\end{abstract}

\maketitle

\section{Introduction}
Let $\mathbb{D}$ be the open unit disc in the complex plane and let $\partial \mathbb{D}$ be the unit circle. With $dA$, respectively $dm$, we denote the normalized Lebesgue measure in $\mathbb{D}$, respectively in $\partial \mathbb{D}$. For $0<p<\infty$ let $\mathbb{H}^p$ be the Hardy space of analytic functions $F$ in $\mathbb{D}$ such that
\[
\|F\|_p^p = \sup_{0< r<1} \int_{\partial \mathbb{D}} |F(r \xi)|^p dm (\xi) < \infty,  
\]
and let $\mathbb{H}^\infty$ be the space of bounded analytic functions $F$ in $\mathbb{D}$ with 
$
\|F\|_\infty = \sup \{|F(z)| : z \in \mathbb{D}\}.
$
Any function $F \in \mathbb{H}^p$ with $0<p \leq \infty$ has radial limit, denoted by $F(\xi)$, at $m$-almost every point $\xi\in \partial \mathbb{D}$ and factors as $F=BE$ where $B$ is an inner function and $E$ an outer function. We recall that $B$ is called inner if $B \in \mathbb{H}^\infty$ and $|B(\xi)|=1$ for $m$-almost every $\xi \in \partial \mathbb{D}$ and that the outer function $E$ can be written as
\[
E(z) = \exp \Big( \int_{\partial \mathbb{D}} \frac{\xi + z}{\xi - z} h (\xi) dm (\xi) \Big), \quad z \in \mathbb{D}, 
\]
where $h$ is an integrable function in the unit circle. Actually $\log |E(\xi)| = h(\xi)$ for $m$-almost every $\xi \in \partial \mathbb{D}$. See Chapter II of \cite{garnett2006bounded}.

Given an arc $I \subset \partial \mathbb{D}$ of normalized length $ m(I) = |I|$, let $Q=Q(I) = \{z \in \mathbb{D} : |z| \geq 1- |I| , z/|z| \in I \}$ be the Carleson square based at $I$. It is customary to denote $\ell (Q) = |I|$. A finite positive Borel measure $\mu$ in $\mathbb{D}$ is called a Carleson measure if there exists a constant $C>0$ such that  
\[
\int_{\mathbb{D}} |F(z)|^pd\mu(z) \leq C \|F\|_p^p ,
\]
for any $F \in \mathbb{H}^p$. A celebrated result of Carleson says that $\mu$ is a Carleson measure if and only if there exists a constant $C_1 >0$ such that $\mu (Q) \leq C_1 \ell (Q)$ for any Carleson square $Q$. A Carleson measure $\mu$ is called a vanishing Carleson measure if
\[
\frac{\mu(Q)}{ \ell (Q)} \to 0 \, \, \text{as} \, \, \ell (Q) \to 0. 
\]
The average of an integrable function $h$ over an arc $I\subset \partial \mathbb D$, is denoted by  
\[
h_I = \fint_I h(\xi) dm(\xi) = \frac{1}{m(I)} \int_{I} h(\xi) dm(\xi) .
\]
An integrable function $h$ in $\partial \mathbb{D} $ is in BMO if 
\[
\|h\|_{\BMO} = \sup \fint_{I} |h(\xi) - h_I| dm(\xi) < \infty , 
\]
where the supremum is taken over all arcs $I \subset \partial \mathbb{D}$. The subspace of functions $h \in \BMO$ such that 
\[
\fint_{I} |h(\xi) - h_I| dm(\xi) \to 0\, \, \text{as} \, \, m(I) \to 0
\]
is denoted by $\VMO$ and coincides with the closure in the $\BMO$ semi-norm of the continuous functions on $\partial \mathbb{D}$. Given an integrable function $h$ in $\partial \mathbb{D}$ we denote by $h(z)$ with $z \in \mathbb{D}$, its harmonic extension to $\mathbb{D}$. Functions in $\BMO$ and Carleson measures are intimately related. Indeed, an integrable function $h$ is in $\BMO$, respectively in $\VMO$, if and only if $|\nabla h(z) |^2 (1-|z|^2) d A(z)$ is a Carleson, respectively vanishing Carleson, measure. See Chapter VI \cite{garnett2006bounded} for all these well known results.

Our work is inspired by the beautiful article \cite{Wolff} of T. Wolff who considered the algebra $\QA$ of bounded analytic functions whose boundary values are in $\VMO$. He proved the following deep result.
\begin{foundationaltheorem}{\cite[Theorem 1]{Wolff}}\label{Wolff theorem}
Given any bounded function $f$ in $\partial \mathbb{D}$, there exists an outer function $E \in \QA$ 
such that $Ef \in \VMO$.
 \end{foundationaltheorem}

We now state our main result. 

\begin{theorem}\label{Theorem 1}

\begin{itemize}
   
    \item[(a)] 
Let $\mu$ be a Carleson measure in ${\mathbb{D}}$. Then, there exists an outer function 
$E \in  \QA$ with $\log|E| \in \VMO$ such that $|E|\mu$ is a vanishing Carleson measure. 
 \item[(b)] 
 Let $\mu$ be a finite positive Borel measure in $\mathbb{D}$. Then, there exists an outer function 
 $E \in  \QA$ with $\log|E| \in \BMO$ such that $|E|\mu$ is a vanishing Carleson measure. 
 \end{itemize}
\end{theorem}

Let $\mu$ be a finite positive Borel measure in $\mathbb{D}$. Roughly speaking, for \textit{most} Carleson squares $Q$ we have $\mu(Q) < C (Q) \ell (Q)$ where $C(Q) \to 0$ as $\ell(Q)\to 0$. 
See Proposition \ref{L:prel}. So for \textit{most} Carleson squares, the contribution from $|E|$ is not needed. Consequently, $|E|$ only needs to be small on \textit{few} Carleson squares, offering the flexibility to construct $E \in \QA$ such that $|E| \mu$ is a vanishing Carleson measure. The proof of Theorem \ref{Theorem 1} uses a decomposition of the measure $\mu$, stopping time arguments  yielding nested families of dyadic arcs on $\partial \mathbb D$ and a construction of certain $\BMO$ functions due to J. Garnett and P. Jones (\cite{Garnett-jones}). 
 
Theorem \ref{Theorem 1} is sharp in several different senses. First, in the conclusion, the vanishing Carleson measure condition can not be replaced by a stronger condition of the same sort. Actually for any increasing function $\omega : [0,1] \to [0, \infty)$ with $\omega(0)=0$, there exists a Carleson measure $\mu$ such that for any outer function $E \in H^\infty$ we have
\[
\limsup_{\ell(Q)\to 0}\frac{\int_{Q}|E(z)|d\mu(z)}{\ell(Q)\omega(\ell(Q))}=\infty . 
\]
See Proposition \ref{Prop:optimality vanishing carleson}. Second, one can not replace $\QA$ by the disk algebra $\A$ of analytic functions in $\mathbb{D}$ which extend continuously to $\overline{\mathbb D}$. Actually a Carleson measure $\mu$ will be constructed for which $|E|\mu$ fails to be a vanishing Carleson measure for any non trivial $E \in \A$. See Proposition \ref{Optimality not E in A}. Finally, in part $b)$ one can not  have the sharper condition $\log|E| \in \VMO$, as it is in part $a)$. See Proposition \ref{Missing optimal}. 

Theorem \ref{Theorem 1} has applications in three different contexts. First, using Theorem \ref{Theorem 1} one can prove Theorem \ref{Wolff theorem} of T. Wolff.
The second application of Theorem \ref{Theorem 1} concerns critical points of functions in Hardy spaces. Let $\BMOA$, respectively $\VMOA$, be the space of functions $F \in \mathbb{H}^2$ whose boundary values $F(\xi)$ with $\xi \in \partial \mathbb{D}$, are in $\BMO$, respectively in $\VMO$. 
W. Cohn proved in \cite[Theorem 1]{Cohn} that given $0< p < \infty$ and $F \in \mathbb{H}^p$, there exists $G \in \BMOA$ such that the zeros (and the multiplicities) of $F'$ and $G'$ coincide. Later D. Kraus in \cite[Theorem 1.1]{Kraus2013} proved that given $0< p<\infty $ and $F \in \mathbb{H}^p$, there exists a Blaschke product $B$ such that the zeros (and the multiplicities) of $F'$ and $B'$ coincide. Using Theorem \ref{Theorem 1} and a beautiful technique developed by D. Kraus in \cite{Kraus2013} we prove the following result. 

\begin{theorem}\label{Theorem 3}
Let $0< p\leq \infty$ and $F \in \mathbb{H}^p$. Then there exists a function $G \in \QA$ such that the zeros (and the multiplicities) of $F'$ and $G'$ coincide.
\end{theorem}

Our last application of Theorem \ref{Theorem 1} concerns generalized Volterra integral operators. Given an analytic function $G$ in $\mathbb D$, the action of the generalized Volterra operator $T_G$ on the analytic function $F$ is defined as
\[
T_G (F) (z)=\int_0^z F(w) G'(w)dw, \quad  z\in \mathbb{D}.
\]
Notice that if $G \in \BMOA$, then $T_G\colon \mathbb{H}^\infty \to \BMOA$ is continuous, see \cite[Proposition 1]{Aleman01101995}. As expected $T_G\colon \mathbb{H}^\infty \to \BMOA$ is compact if and only if $G \in \VMOA$. We will apply Theorem \ref{Theorem 1} to obtain the following result.

\begin{theorem}\label{Theorem 4}
Let $G\in \BMOA$ be non constant. Then
$T_G : \mathbb{H}^\infty \to \BMOA$ is not bounded from below and $T_G (\mathbb{H}^\infty)$ is not closed in $\BMOA$.
\end{theorem}

The rest of the paper is organized as follows. Section 2 contains auxiliary results that are used in the proof of Theorem \ref{Theorem 1} which is given in Section 3. In Section 4, Theorem \ref{Theorem 1} is applied to prove Theorems \ref{Wolff theorem}, \ref{Theorem 3} and \ref{Theorem 4}. The sharpness of Theorem \ref{Theorem 1} is discussed in the last section. 

As usual, the notation $A \lesssim B$ means that there exists a universal constant $C>0$ such that $A \leq C B$. 

It is a pleasure to thank O. Ivrii, D. Kraus and O. Roth for several helpful discussions. 

\medskip
\section{Auxiliary Results}
Our first result states that any finite positive Borel measure in $\mathbb{D}$ can be written as the sum of two measures which have small mass on a certain sequence of annuli.


\begin{Lemm}\label{Lemma 2.1}
  Let $\mu$ be a finite positive Borel measure in $\mathbb D$. Let $\{\varepsilon_n\}$ be a decreasing sequence of positive numbers tending to zero. Then there exists an increasing sequence $\{r_n \}$ with $0\leq r_n<1$, $n=0,1,\dots$ and two positive measures $\mu_1$, $\mu_2$ with $\mu=\mu_1+\mu_2$ such that
  \begin{equation}\label{E:propmu1}
       \mu_1 \left\lbrace  z \in \mathbb{D} : |z|>r_{2n+1}\right\rbrace \leq \varepsilon_{2n+1}(1-r_{2n+1}) 
  \end{equation}
  and
  \begin{equation}\label{E:propmu2}
        \mu_2\left\lbrace   z \in \mathbb{D} :  |z|>r_{2n}\right\rbrace\leq \varepsilon_{2n}(1-r_{2n}), 
  \end{equation}
  for $n=0,1,\dots$. Moreover $\{r_n\}$ can be chosen of the form $r_n = 1- 2^{-N(n)}$ for some integer $N(n) \geq 0$, $n=0,1,\ldots$ and $\sum (1-r_n) < \infty$.
\end{Lemm}
\begin{proof}
By induction one can define an increasing sequence $\{r_n \}$ with $r_0 =0$ and $0 \leq r_n < 1$, $n=1,2,\ldots$, such that 
\[
\mu\{z\in\ \mathbb{D}\colon \  |z|\geq r_{n+1}\}\leq \varepsilon_n (1-r_n) , \quad n=0,1,2, \ldots .
\]
It is clear that $\{ r_n \}$ can be taken as described in the last part of the statement. Let ${ \mathbbm{1} }_n$ be the indicator function of the annulus $\{z \in \mathbb{D} : r_n \leq |z| < r_{n+1} \}$. Define the two measures  $\mu_1,\mu_2$ as 
\[
\mu_1=\mu \sum_{n=0}^\infty  { \mathbbm{1} }_{2n} \text{ and } \mu_2 = \mu \sum_{n=0}^\infty { \mathbbm{1} }_{2n+ 1}. 
\]
We notice that $\mu=\mu_1+\mu_2$. Moreover for $n=0,1, \dots$, we have that
\[
\mu_1\{|z|\geq r_{2n+1}\}=\mu_1\{|z|\geq r_{2n+2}\}\leq \mu\{|z|\geq r_{2n+2}\}\leq \varepsilon_{2n+1}(1-r_{2n+1})
\]
and
\[
\mu_2\{|z|\geq r_{2n}\}=\mu_2\{|z|\geq r_{2n+1}\}\leq \mu\{|z|\geq r_{2n+1}\}\leq \varepsilon_{2n}(1-r_{2n}).
\]
\end{proof}
The proof of Theorem \ref{Theorem 1} uses a beautiful construction due to P. Jones and J. Garnett (\cite{Garnett-jones}) of certain functions in $\BMO$ supported in a given arc which are large on certain subsets of the arc.  
A Lipschitz function $a\colon \partial \mathbb D \to \mathbb{R}$ is called $B$-adapted to the arc $I\subset \partial \mathbb D$ if the following three conditions hold: the support of $a$ is contained in the dilated arc $3I$; $\sup |a| \leq 1$ and $|\nabla a_j(\xi)|\leq B/|I|$ for any $\xi \in \partial \mathbb{D}$.

\begin{Lemm}\label{Lemma 2.2} {\cite{Garnett-jones}}
Let $\{I_j\}$ be a sequence of arcs in $\partial \mathbb{D}$. Assume that there exists a constant $C_1 >0$ such that for any arc $ I \subseteq \partial \mathbb{D}$ we have 
\begin{equation}\label{E:packing}
    \sum_{I_j\subset I} |I_j|\leq C_1 |I|. 
\end{equation}
Let $a_j$ be a B-adapted function to the arc $I_j$ for $j=1,2, \dots$. Then $\sum_j a_j\in  \BMO$
and 
\[
\bigg\|\sum_j a_j\bigg\|_{\BMO}\lesssim C_1B.
\]
\end{Lemm}
For the proof of Lemma \ref{Lemma 2.2}, we refer to Lemma 2.1 in \cite{Garnett-jones} (see Lemma 3.2 of \cite{NICOLAU200021} for a $\VMO$ version).
The following result may be known but, since it is not clearly stated in the literature, we provide a short proof based on an idea in Lemma 1.2 in \cite{Wolff}. 

\begin{corollary}\label{Corollary 2.3}
Let $\{I_j\}$ be a sequence of arcs in $\partial \mathbb D$ with $\sum_{j} |I_j|<\infty$. Then, there exists a positive function $f \in \VMO$ such that 
\[
\lim_{j \to \infty} \fint_{I_j}f  dm  =+\infty.
\]
\end{corollary}
\begin{proof}
We use Lemma 2.2 of \cite{Garnett-jones}, which says that given a measurable set $E\subset \partial \mathbb D$ there exists a positive function $h$ with $\|h \|_{\BMO} \leq 1$ such that 
\[
-\log m (E) \lesssim h (\xi) , \quad \text{ for all }\xi \in E.
\]
Moreover, if $E$ is a finite union of arcs, then $h$ may be taken in $C^\infty(\partial \mathbb D)$.
Since $\sum_j |I_j|<\infty$, the collection $\{I_j\}$ can be split as $\{I_j \} = \cup_{n \geq 1} \mathcal{A}_n$ where $\mathcal{A}_n$ is a collection of finitely many arcs which satisfies  
\[
\sum_{I \in \mathcal{A}_n } |I| \lesssim e^{-n^3} , \quad n=1,2, \ldots
\]
Let $f_n$ be a positive smooth function in $\partial \mathbb{D}$ satisfying $\|f_n\|_{\BMO}\lesssim 1$ and 
\[
f_n (\xi) \geq n^3, \quad \text{ for all }\xi \in \bigcup_{I \in \mathcal{A}_n} I .
\]
We set $f=\sum_n f_n/n^2$. Then $ f \in \VMO$  and  
\[
\lim_{j \to \infty} \frac{1}{|I_j|}\int_{I_j}f  dm =\infty.
\]
\end{proof}    

Let $\mathcal{D}$ denote the family of dyadic arcs of the unit circle. The corresponding family of dyadic Carleson squares is defined as $\{Q(I) : I \in \mathcal{D} \}$. Note that two dyadic Carleson squares are either disjoint or one is contained into the other. Given a Carleson square $Q=Q(I)$ where $I \subset \partial \mathbb{D}$ is an arc centered at the point $\xi \in \partial \mathbb{D}$, consider $z_Q = (1- \ell(Q)) \xi$. The next preliminary result will be needed in the proof of Theorem \ref{Wolff theorem}.
\begin{Lemm}\label{Lemma 2.4}
Let $\mu$ be a Carleson measure in $\mathbb D$. Given $\varepsilon>0$ consider the collection $\mathcal{A}=\mathcal{A}(\varepsilon)$ of Carleson squares $Q$ such that $\mu(Q)\geq \varepsilon \ell(Q)$. 
Let $E \in \VMOA$ such that $|E|\mu$ is a vanishing Carleson measure. Then 
\[
\lim_{Q \in \mathcal{A},\, \ell(Q)\to 0}|E(z_Q)|=0.
\]
\end{Lemm}
\begin{proof}
Given a Carleson square $Q$ denote $I(Q)= \overline{Q} \cap \partial \mathbb{D}$. Fixed a constant $\eta >0$ and a Carleson square $Q \in \mathcal{A} (\varepsilon)$, consider the family $\{Q_j \}$ of maximal dyadic Carleson squares contained in $Q$ such that
\[
\sup \{|E(w) - E(z_Q)| : w \in T(Q_j)\} \geq \eta . 
\]
Here $T(Q_j) = \{z \in Q_j : 1-|z| \geq \ell (Q_j) / 2\}$ is the top part of $Q_j$. Since $E\in \VMOA$, we have
\begin{equation}
    \label{vmo}
    \frac{1}{\ell(Q)} \sum \ell (Q_j) \to 0 \text{    as    } \ell (Q) \to 0. 
\end{equation}
Consider the region $R = R(Q, \eta) = Q \setminus \cup Q_j$. Since $\mu $ is a Carleson measure and $Q \in \mathcal{A} (\varepsilon)$, from \eqref{vmo} we deduce that 
\begin{equation}
    \label{vmo1}
    \frac{ \mu (R)}{\mu(Q)}  \to 1 \text{    as    } \ell (Q) \to 0. 
\end{equation}
Moreover, by construction we have that
\begin{equation*}
    \sup \{|E(w) - E(z_Q)| : w \in R\} \leq \eta . 
\end{equation*}
Hence 
$$
\int_Q |E| d \mu \geq \int_R |E| d \mu \geq \frac{1}{2} (|E(z_Q)| - \eta ) \varepsilon \ell (Q) , 
$$
if $\ell (Q)$ is sufficiently small. Since $|E|\mu$ is a vanishing Carleson measure and $\eta >0$ can be taken arbitrarily small, we deduce that $|E(z_Q)| \to 0$ as $\ell (Q)$ tends to $0$. 
\end{proof}

Next result will be used in the proof of Theorem \ref{Theorem 3}. It can be understood as a hyperbolic analogue of the classical fact that a harmonic function $u$ in $\mathbb{D}$ such that $|\nabla u (z)|^2 (1-|z|^2) dA(z)$ is a Carleson measure, has boundary values in BMO. 

\begin{Lemm}
     \label{nou}
     Let $F$ be an analytic self-mapping of the unit disc. 

     (a) Assume that
     \begin{equation}
         \label{hipotesis}
         \frac{|F' (z)|^2 (1-|z|^2)}{(1 - |F(z)|^2)^2}  dA(z)
     \end{equation}
 is a Carleson measure. Then $\log (1- |F|^2) \in \BMO $. 
     
     (b) Assume that
     \begin{equation}
         \label{hipotesis2}
         \frac{|F' (z)|^2 (1-|z|^2)}{(1 - |F(z)|^2)^2}  dA(z)
     \end{equation}
 is a vanishing Carleson measure. Then $\log (1- |F|^2) \in \VMO$ and consequently, 
 $$
 \sup_{0<r<1} \int_{\partial \mathbb{D}} \frac{dm (\xi)}{(1-|F(r \xi)|)^s} < \infty,
 $$
 for any $0<s<\infty$.
 \end{Lemm}

 \begin{proof}
(a) Let $u(z) = -\log (1-|F(z)|^2)$, $z \in \mathbb{D}$. By Schwarz's Lemma there exists a constant $C_1 >0$ such that
\begin{equation}
    \label{Schwarz}
    \sup \{|u(z) - u(w)| : z \in T(Q), w \in T(Q_1) \} \leq C_1 ,
\end{equation}
for any pair of Carleson squares $Q_1 \subset Q$ with $\ell (Q_1) = \ell (Q) / 2$. Let $K$ be the Carleson norm of the Carleson measure in \eqref{hipotesis}. Let $C> 2 C_1 + 2K$ be a large constant to be determined later. 

Let $I$ be an arc of the unit circle and consider the dyadic decomposition of $Q(I)$. We now use a stopping time argument. Let $\mathcal{G}_1$ be the collection of maximal dyadic Carleson squares $Q_j^{(1)} \subset Q(I)$ such that
\begin{equation}
    \label{generacio1}
    \sup \{|u(z) - u(z_{Q(I)})| : z \in T(Q_j^{(1)})\} \geq C. 
\end{equation}
The maximality and \eqref{Schwarz} give that 
$$
C- C_1 \leq |u(z_{Q_j^{(1)}}) - u(z_{Q(I)})| \leq C + C_1 . 
$$
We continue by induction. More concretely, assume that the collection $\mathcal{G}_{n-1} = \{Q_l^{(n-1)} : l =1,2,\ldots \}$ has been defined. For each $Q_l^{(n-1)} \in \mathcal{G}_{n-1}$ consider the collection $\mathcal{G}_{n} (Q_l^{(n-1)})$ of maximal dyadic Carleson squares $Q_j^{(n)} \subset Q_l^{(n-1)}$ such that
\begin{equation}
    \label{genn}
\sup \{|u(z) - u(z_{Q_l^{(n-1)}})| : z \in T(Q_j^{(n)}) \} \geq C. 
\end{equation}
 The collection $\mathcal{G}_n$ is defined as 
$$
\mathcal{G}_n = \cup_l \, \mathcal{G}_{n} (Q_l^{(n-1)}). 
$$
 As before the maximality and \eqref{Schwarz} give that 
     $$
   C- C_1 \leq   |u(z_{Q_j^{(n)}}) - u(z_{Q_j^{(n-1)}})| \leq C + C_1 . 
     $$
Observe that 
\begin{equation}
    \label{boundoutside}
    |u(\xi)- u(z_I)|\leq (C+ C_1)n, \quad \xi \in   I \setminus \cup_j \overline{Q_j^{(n)}}, \quad n=1,2,\ldots.    
\end{equation}

Let $Q \in \mathcal{G}_{n-1}$. Consider the region
$
\Omega = Q \setminus \cup Q_j^{(n)},
$
where the union is taken over all Carleson squares $Q_j^{(n)} \in \mathcal{G}_{n} (Q) $. Note that 
\begin{equation}
    \label{laplacian}
    \Delta u (z)=  \frac{|F' (z)|^2 }{(1 - |F(z)|^2)^2}, \quad z \in \mathbb{D}
\end{equation}
and 
\begin{equation}
    \label{gradient}
    |\nabla u (z)|^2 \leq \frac{|F' (z)|^2 }{(1 - |F(z)|^2)^2} , \quad z \in \mathbb{D}. 
\end{equation}
These facts follow from direct calculations and have been recently used in \cite{ivrii2024analytic}, \cite{bampouras2025inner} and \cite{ivrii2025analytic}. Since $\Delta (u-u(z_Q))^2 = 2 |\nabla u|^2 + 2 (u- u(z_Q)) \Delta u $ and $|u(z) - u(z_Q)| \leq C$ for any $z \in \Omega$, using \eqref{laplacian} and \eqref{gradient} we deduce that
\begin{equation}
    \label{estim}
    \Delta (u(z) -u(z_Q))^2 \leq   \frac{2(C+1) |F' (z)|^2 }{(1 - |F(z)|^2)^2} , \quad z \in \Omega . 
\end{equation}
Apply Green's Formula to the functions $(u(z) - u(z_Q))^2$ and $ \log |z|$, Then the estimate \eqref{estim} gives 
\begin{align}
\begin{split}
    \label{inter}
    & | \int_{\partial \Omega} (u(z) - u(z_Q))^2 \,  \partial_{n} \log |z| ds (z) -  \int_{\partial \Omega} \log |z| \, \partial_{n} (u(z) - u(z_Q))^2  ds (z) | \lesssim  \\
    & \lesssim \int_\Omega \frac{ 2(C+1) |F' (z)|^2 (1-|z|^2)}{(1 - |F(z)|^2)^2} dA(z) \leq 2 (C+1) K \ell (Q).  
    \end{split}
    \end{align}
    Note that $|\nabla (u(z) - u(z_Q))^2| \leq 2 C |\nabla u (z)|$ for any $z \in \Omega$ and that $(1-|z|^2)|\nabla u (z)| \leq 1$ for any $z \in \mathbb{D}$. We deduce that
\begin{equation}
    \label{priint}
    \int_{\partial \Omega} | \log |z| \, \partial_{n} (u(z) - u(z_Q))^2 | ds (z) \lesssim C \ell (Q). 
\end{equation}
Note that $\partial_{n} \log |z|$ is supported on the circular parts of the boundary of $\Omega$, where it has values $\pm 1/|z|$. Since $|u -u(z_Q)| \leq C_1$ on $T(Q)$ and $|u(z) - u(z_Q)| > C- C_1$ for any $z \in T(Q_j^{(n)})$, from \eqref{inter} and \eqref{priint}, we deduce
\begin{equation}
    \label{ifinal1}
    C^2 \sum_j \ell (Q_j^{(n)}) \lesssim CK \ell (Q)
\end{equation}
Fix $C$ sufficiently large such that
\begin{equation}
    \label{ifinal}
\sum_j \ell (Q_j^{(n)}) \lesssim \frac{K}{C} \ell (Q) \leq \frac{1}{2} \ell (Q). 
\end{equation}
Iterating this estimate we obtain
\begin{equation}
    \label{iifinal}
\sum_j \ell (Q_j^{(n)}) \lesssim \frac{1}{2^n} |I|, \quad n=1,2,\ldots . 
\end{equation}

Next we show that $u \in \BMO$. Fix $\lambda > 2C$ and let $n$ be the integer part of $\lambda  / (C+C_1)$. Note that \eqref{boundoutside} gives that
$$
\{\xi \in I : |u(\xi) - u(z_I)| > \lambda\} \subset \cup_j \overline{Q_j^{(n)}} \cap I. 
$$
Then, estimate \eqref{iifinal} gives that
$$
m ( \{\xi \in \partial \mathbb{D}: |u(\xi) - u(z_I)| > \lambda\} ) \leq \frac{1}{2^n} |I|. 
$$
This implies that
$$
\frac{1}{|I|} \int_I |u(\xi) - u(z_I)| dm (\xi) = \frac{1}{|I|} \int_0^\infty m (\{\xi \in I : |u(\xi) - u(z_I)| > \lambda\}) d \lambda 
$$
is bounded by a universal constant independent of $I$. This finishes the proof of (a). 

(b) The proof of (b) only requires minor modifications. Actually one only needs to observe that if $|I|$ is sufficiently small, the constant $K$ in \eqref{ifinal1} can be taken also small. This allows to fix also $C>0$ such that $K/C$ is also small. Hence given $\varepsilon >0$, if $|I|$ is sufficiently small, one can replace the factor $1/2$ in \eqref{ifinal} by $\varepsilon$. This gives that $\log (1-|F|^2) \in \VMO$. Last assertion in (b) follows from the well known fact that $u \in \VMO$ implies that $e^{su}$ is integrable for any $s>0$. 

 \end{proof}

\medskip

\section{Proof of Theorem \ref{Theorem 1}}

We now prove our main result.
\begin{proof}[Proof of Theorem \ref{Theorem 1}]

The proof is organized in three steps. 

\textbf{1. Splitting the measure.} Let $\mu$ be a finite Borel measure on $\mathbb{D}$. Let $\{\varepsilon_n\}$ be a decreasing sequence of positive numbers tending to $0$. 
We apply Lemma \ref{Lemma 2.1} to find two measures $\mu_1,\mu_2$ which satisfy \eqref{E:propmu1} and \eqref{E:propmu2} respectively.
For $n\geq 0$, we pick the maximal dyadic squares $\{Q^{n}_k : k=1,2, \ldots\}$ with $
1-r_{2n+3}<\ell(Q^{n}_k)\leq 1-r_{2n+1}$ such that
\begin{equation}
    \label{tag}
    \frac{\mu_1(Q^{n}_k)}{\ell(Q^{n}_k)}\geq \varepsilon_{2n+1} . 
\end{equation}
We notice that if $Q$ is a Carleson square with $\ell(Q)=1-r_{2n+1}$, then
\[
\frac{\mu_1 (Q)}{\ell(Q)}< \varepsilon_{2n+1}.
\]
Indeed, applying \eqref{E:propmu1}, we have 
\begin{equation}\label{E:not Heavithemaximal}
    \frac{\mu_1(Q)}{\ell(Q)}\leq \frac{\mu_1(\{z \in \mathbb{D}\colon \ |z|\geq r_{2n+1}\})}{1- r_{2n+1}} \leq \varepsilon_{2n+1}.
\end{equation}
Since $\{Q^n_k : k=1,2,\ldots \}$ are pairwise disjoint, equation \eqref{tag} gives that 
\[
\sum_{k}\ell(Q^n_k)\leq \frac{1}{\varepsilon_{2n+1}}\mu_1(\{|z|\geq r_{2n+1}\}) \leq 1- r_{2n+1} 
\]
and we have 
\begin{equation}\label{sum Inj}
 \sum_n \sum_{k}\ell(Q^n_k)\leq \sum_n (1-r_{2n+1})< \infty. 
\end{equation}
We split each arc $I^n_k= \overline{Q^n_k } \cap \partial \mathbb D$ into finitely many smaller pairwise disjoint subarcs $\{ J^n_{k,j} : j =1, 2, \ldots \}$ such that 
\[
|J^n_{k,j}|=1-r_{2n+5}, \quad j=1,2,\ldots 
\]
Due to \eqref{sum Inj}, we have that 
\[
\sum_{n,k, j} |J^n_{k,j} | = \sum_{n}\sum_{k} |I_k^n|<\infty.
\]
A similar construction is applied to the measure $\mu_2$. 

\textbf{2. Proof of (a)}. 
For $i=1,2$ we will construct an outer function $E_i \in \QA$ with $\log |E_i| \in \VMO$ such that $|E_i| \mu_i$ is a vanishing Carleson measure. Once this is done the result will follow easily. We will explicitly describe $E_1$. The function $E_2$ is constructed using the same procedure. 

Apply Corollary \ref{Corollary 2.3} to find a positive function $f \in \VMO $ such that
\begin{equation}\label{E:applwolf}
 \lim_{ | J^n_{k,j} |\to 0}\fint_{J^n_{k,j}}f  dm=+\infty.   
\end{equation}
Consider the outer function $E_1$ defined by $\log |E_1 (\xi)| = - f (\xi)$, $\xi \in \partial \mathbb{D}$. Note that $\| E_1 \|_\infty \leq 1$. Since $f \in \VMO$ we have that $E_1 \in \QA$. Next we will show that $|E_1| \mu_1$ is a vanishing Carleson measure.  

We first argue with dyadic Carleson squares. Given a dyadic Carleson square $Q$ fix $n$ such that $1-r_{2n+3} <  \ell(Q)\leq 1-r_{2n+1}$. If $Q$ is not contained in any of the $\{Q^n_k : k=1,2,\ldots \}$, we have that $\mu_1 (Q) \leq \varepsilon_{2n+1} \ell (Q)$ and 
\[
\int_{Q}|E_1(z)|d\mu_1(z) \leq \mu_1 (Q ) \leq \varepsilon_{2n+1}\ell(Q).
\]
If $Q \subset Q^n_k$ for some $k$, then
\[
\int_Q |E_1(z)|d\mu_1(z)=\int_{Q\cap \{|z|\geq r_{2n+3}\}}|E_1(z)|d\mu_1(z)+\int_{Q\cap \{|z|< r_{2n+3}\})}|E_1(z)|d\mu_1(z). 
\]
Applying \eqref{E:propmu1}, we have 
\begin{align*}
    \int_{Q\cap \{|z|\geq r_{2n+3}\}}|E_1(z)|d\mu_1(z)&\leq  \mu_1(Q\cap \{|z|\geq r_{2n+3}\})
\leq  \varepsilon_{2n + 3} (1 - r_{2n + 3})  \leq  \varepsilon_{2n+3} \ell(Q). 
\end{align*}
Fix $z \in Q \cap \{z \in \mathbb{D} : |z| < r_{2n+3} \}$. Consider the arc $I(z) \subset \partial \mathbb{D}$ centered at $z/|z|$ of length $1-|z|$. Note that $|I(z)| \geq 1 - r_{2n +3}$. Since $Q \subset  Q^n_k $ we deduce that 
$$
\sum_{i : J^n_{k,i} \subset I(z)} |J^n_{k,i}| \geq \frac{|I(z)|}{4}.  
$$
Hence, by \eqref{E:applwolf}, given $\varepsilon >0$ we have 
$
|E_1 (z)| < \varepsilon 
$ if $n$ is sufficiently large. Thus 
\[
\int_{Q\cap \{|z|< r_{2n+3}\})}|E_1(z)|d\mu_1(z)\leq \varepsilon \mu_1(Q),  
\]
if $n$ is sufficiently large. Hence, given $\varepsilon >0$ there exists $\delta >0$ such that  
$$
\int_Q |E_1| d \mu_1 \leq \varepsilon \ell (Q)
$$
when $Q$ is a dyadic Carleson square with $\ell (Q) < \delta$. Since for any Carleson square $Q$ one can find two dyadic Carleson squares $Q_1,Q_2 $ such that $Q \subset Q_1 \cup Q_2$ and $\ell (Q_i) \leq 2 \ell (Q)$, $i=1,2$, we deduce that $|E_1|  \mu_1$ is a vanishing Carleson measure. 

We repeat the above construction for $\mu_2$ and we find another outer function $E_2$ with $\log |E_2| \in \VMO$ such that $|E_2|\mu_2$ is a vanishing Carleson measure. Now $E=E_1 E_2 \in \QA$ satisfies that $|E| \mu$ is a vanishing Carleson measure. Moreover $\log|E| = \log |E_1| + \log |E_2| \in \VMO$.  

\textbf{3. Proof of (b)}. We now prove part (b) of Theorem \ref{Theorem 1}. Let $\mu = \mu_1 + \mu_2$ be the decomposition of Step 1. The Carleson squares $Q^n_k$ of Step 1 will now be denoted as $Q^n_k = Q^{n,0}_k$, $k=1, \ldots$; $n=1, \ldots$. Note that \eqref{tag} and the maximality gives that $\mu_1 (Q^{n,0}_k) \leq 2 \varepsilon_{2n + 1} \ell (Q^{n,0}_k) $. We will now use a stopping time argument in each $Q^{n,0}_k$. Fix $n$ and $k$ and pick the maximal dyadic Carleson squares $\{Q^{n,1}_j : j=1,2,\ldots \}$ contained in $Q^{n,0}_k$ such that 
\[
\frac{\mu_1(Q^{n,1}_j ) }{\ell(Q^{n,1}_j)}\geq 10\cdot \varepsilon_{2n+1}.
\]
Note that the maximality gives that $\mu_1(Q^{n,1}_j ) \leq 20\cdot \varepsilon_{2n+1} \ell(Q^{n,1}_j)  .$ We continue by induction, that is, if $i>1$ is an integer and we have constructed a Carleson square $Q^{n,i-1}_j$ such that 
$10^{i-1} \cdot \varepsilon_{2n+1} \ell(Q^{n,i-1}_j) \leq \mu_1(Q^{n,i-1}_j ) \leq 2 \cdot 10^{i-1} \cdot \varepsilon_{2n+1} \ell(Q^{n,1}_j)$ , we consider the maximal dyadic Carleson boxes $\{Q^{n,i}_l : l = 1,2, \ldots \}$  contained in $Q^{n,i-1}_j$ such that   
\[
\frac{\mu_1(Q^{n,i}_l ) }{\ell(Q^{n,i}_l)}\geq 10^i \cdot \varepsilon_{2n+1}.
\]
As before, the maximality gives 
\[
\frac{\mu_1(Q^{n,i}_l)}{\ell(Q^{n,i}_l)}\leq 2\cdot 10^i \cdot \varepsilon_{2n+1} , \quad l=1,2,\ldots . 
\]
We denote $I^{n,i}_l={\overline{ Q^{n,i}_l} }\cap \partial \mathbb D.$ Since 
\begin{equation*}
\sum_{l:\ I^{n,i}_l \subset I^{n,i-1}_j} |I^{n,i}_l| \leq \frac{1}{10^i \cdot \varepsilon_{2n+1}} \sum_{l:\ I^{n,i}_l \subset I^{n,i-1}_j}  \mu_1(Q^{n,i}_l) \leq  \frac{\mu_1(Q^{n,i-1}_j)}{10^i \cdot \varepsilon_{2n+1}} \leq \frac{1}{5}\ell(Q^{n,i-1}_j), \\
\end{equation*}
we obtain that
\begin{equation}
\label{E:packing proof}
    \sum_{l:\ I^{n,i}_l \subset I^{n,i-1}_j}|I^{n,i}_k|\leq \frac{1}{5}|I^{n,i-1}_j|, \quad i=1,2, \ldots ; j=1,2, \ldots ;  n=0, 1, \ldots . 
\end{equation}
Consequently, iterating \eqref{E:packing proof}, we have that
\begin{equation}
\label{itera}
  \sum_{l:\ I^{n,i}_l\subset I^{n,0}_k}|I^{n,i}_l|\leq \frac{1}{5^i}|I^{n,0}_k| , \quad i,k=1,2,\ldots ; n=0,1, \ldots.  
\end{equation}

For each $n,i,l$ we pick a smooth $B$-adapted function $a^{n,i}_l$ to the arc $I^{n,i}_l$ such that $a^{n,i}_l (\xi)=1$ if $\xi \in I^{n,i}_l$. By Lemma \ref{Lemma 2.2} we have $\|\sum_l a^{n,i}_l  \|_{\BMO} \lesssim B/5^i$. The function 
\[
h_n=\sum_{i=0}^\infty  \sum_{l=1 }^\infty a^{n,i}_l
\]
belongs to $\BMO $ and $\|h_n\|_{\BMO} \lesssim B$. 
Moreover, applying \eqref{itera}, \eqref{tag} and \eqref{E:propmu1} we obtain 
\begin{align*}
    \int_{\partial \mathbb{D}}h_n  dm \lesssim&  6 \sum_{i=0}^\infty \sum_{l=1}^\infty |I^{n,i}_l| \lesssim  \sum_{k=1}^\infty |I^{n,0}_k| \lesssim \frac{1}{\varepsilon_{2n+1}}\sum_k \mu_1(Q^{n,0}_k)\\
\leq & \frac{1 }{\varepsilon_{2n+1}} \mu_1(\{z\in \mathbb D\colon \ |z|\geq r_{2n+1}   \}) \leq  1-r_{2n+1} . 
\end{align*}
Let $E_1$ be the outer function given by  
\[
\log |E_1 (\xi)| = -  4 \cdot \log (10) \sum_{n=1}^\infty h_n (\xi), \quad \xi \in \partial \mathbb{D}.  
\]
Note that $\| E_1 \|_\infty \leq 1$. At this point, we verify that $|E_1| \mu_1$ is a vanishing Carleson measure. We first deal with dyadic Carleson squares. Let $Q$ be a dyadic Carleson square. Pick the integer $n$ such that $1-r_{2n+3} < \ell(Q)\leq 1-r_{2n+1}$. If $Q$ lays outside $\bigcup_k Q^{n,0}_k$, then
$$
\frac{\mu_1(Q)}{\ell(Q)}\leq \varepsilon_{2n+1}.
$$
On the other hand, if $Q \subset Q^{n,i-1}_k\setminus \bigcup_l Q^{n,i}_l$ for some integer $i \geq 1$, $k=1,2,\ldots$, then 
$$
\frac{\mu_1(Q)}{\ell(Q)}\leq 10^i \varepsilon_{2n+1}.
$$
Note that $h_n (\xi) \geq i$ for any $\xi \in I^{n,i-1}_k$. Then $\log |E_1 (z)| \leq -i \log(10)$ for any $z \in Q^{n, i-1}_k$ and we obtain
\begin{align*}
    \frac{\int_Q |E_1(z)|d\mu_1(z)}{\ell(Q)}&\leq 10^{-i} \frac{\mu_1 (Q)}{\ell (Q)}
    \leq   \varepsilon_{2n+1}.
\end{align*}
Hence given $\varepsilon >0$ we have proved that
\[
\int_Q | E_1 | d \mu_1 \leq \varepsilon \ell (Q)
\]
if $Q$ is a dyadic Carleson square with $\ell (Q)$ sufficiently small. Since any Carleson square $Q$ is contained in the union of two dyadic Carleson squares $Q_1 , Q_2$ with $\ell (Q_i) \leq 2 \ell (Q)$, $i=1,2$, we deduce that $|E_1| \mu_1$ is a vanishing Carleson measure. 
The same construction applied to $\mu_2$ provides an outer function $E_2 \in H^\infty $ such that $|E_2| \mu_2$ is a vanishing Carleson measure. 

We apply part (a) of Theorem \ref{Theorem 1} to the Carleson measure
\[
|\left(E_1(z)E_2(z)\right)'|^2(1-|z|^2)dA(z)
\]
and find an outer function $F \in \QA$ with $\log|F| \in \VMO$ such that the measure
\[
|F (z) | |\left(E_1(z)E_2(z)\right)'|^2(1-|z|^2)dA(z)
\]
is a vanishing Carleson measure. Note that since $\log |F| \in \VMO$ we have that $F^{1/2} \in \QA$.  Consequently, the function $E= F^{1/2} E_1 E_2$ is an outer function in $\QA$ such that $|E|\mu$ is a vanishing Carleson measure. 

Finally we show that $\log|E| \in \BMO $. It is sufficient to prove that $\log|E_1| \in \BMO$. Since 
\[
-\log|E_1 (\xi)|=\sum_n h_n (\xi) =\sum_n \sum_i \sum_l a^{n,i}_l (\xi), \quad \xi \in \partial \mathbb{D} , 
\]
it is sufficient  to observe that $\{a^{n,i}_l\}_{n,i,l}$ are $B$-adapted functions to the arcs $\{I^{n,i}_l\}$ which satisfy the packing condition \eqref{E:packing}.
\end{proof}

\medskip

\section{Applications}
\subsection{Theorem \ref{Wolff theorem}}
Our first application is a proof of Theorem \ref{Wolff theorem} of T. Wolff using Theorem \ref{Theorem 1}. 
\begin{proof}[Proof of Theorem \ref{Wolff theorem}]
Let $P_z(f)=f(z)$ denote the harmonic extension of $f$ to $\mathbb D$. We apply case (a) of Theorem \ref{Theorem 1} to the Carleson measure $\mu(z)=|\nabla P_z(f)|^2 (1-|z|^2)dA(z)$
to obtain an outer function $E \in \QA$ with $\log|E| \in \VMO$ such that $|E| \mu$
 is a vanishing Carleson measure. 
Let $I \subset \partial \mathbb{D}$ be an arc. We have that
\begin{align*}
   & \fint_{I}|E(\xi)f(\xi)-E(z_I)f(z_I)|dm (\xi) \\& \leq \fint_{I}|E(\xi)-E(z_I)||f(\xi)|dm (\xi)+  \fint_{I}|f(\xi)-f(z_I)||E(z_I)|dm (\xi)\\
    &\leq \|f\|_{\infty}\fint_{I}|E(\xi)-E(z_I)|dm (\xi)+ |E(z_I)|\fint_{I}|f(\xi)-f(z_I)|dm (\xi), \end{align*}
where $z_I=(1-|I|) \xi_I$ and $\xi_I$ is the center of $I$.
Since $E \in \VMOA$, the first integral tends to $0$ as $|I|$ tends to $0$ and 
we only need to show that 
\begin{equation}
    \label{final}
\lim_{|I| \to 0} |E(z_I)|\fint_{I}|f(\xi)-f(z_I)|dm (\xi) = 0.
\end{equation}
By the Cauchy-Schwarz inequality, we have that
\begin{align*}
   & \left( \fint_{I}|f(\xi)-f(z_I)|dm(\xi) \right)^2 \leq \fint_{I}|f(\xi)-f(z_I)|^2dm (\xi) 
   \lesssim  \int_{\partial \mathbb D}|f(\xi)-f(z_I)|^2 P_{z_I}(t)dm (\xi) \\
  &\leq  \int_{\mathbb D}|\nabla f(w)|^2 \frac{(1-|z_I|^2)(1-|w|^2)}{|1-\overline{z_I}w|^2}dA(w) = \int_{\mathbb D}\frac{(1-|z_I|^2)}{|1-\overline{z_I}w|^2}d\mu (w) ,
\end{align*}
where $P_{z_I} $ is the Poisson kernel at the point $z_I$. We notice that   
\[
\frac{1-|z_I|^2}{|1-\overline{w}z_I|^2} \lesssim  \frac{1}{2^{2n}|I|}, \quad w \in 2^nQ(I)\setminus 2^{n-1}Q(I) ,  n\geq 1.
\]
Thus
\begin{align*}
 & \int_{\mathbb D} \frac{(1-|z_I|^2)}{|1-\overline{z_I}w|^2}d\mu (w)    \lesssim \ \frac{\mu (Q(I))}{|I|}  + \sum_{n\geq 1}\frac{\mu (2^{n} Q(I) \setminus 2^{n-1} Q(I) )}{2^{2n} |I|} .   
\end{align*}
Since $\mu$ is a Carleson measure, for any $\varepsilon >0$ we have that
\[
\sum_{n \geq \log(1/\varepsilon)} \frac{1}{2^{2n} |I|} \mu (2^n Q(I) \setminus 2^{n-1} Q(I) )\lesssim \sum_{n \geq \log(1/\varepsilon)}\frac{1}{2^n}\leq \varepsilon
\]
and then
\begin{equation}
    \label{final1}
\fint_{I}|f(\xi)-f(z_I)|dm (\xi) \lesssim \Big( \sum_{n=1}^{\log (1/ \varepsilon)} \frac{\mu (2^n Q(I))}{2^{2n} |I|} + \varepsilon \Big)^{1/2} \leq \Big( \frac{\mu (\varepsilon^{-1} Q(I)) }{|I|} + \varepsilon \Big)^{1/2} . 
\end{equation}
We are now going to show \eqref{final}. Fix $\varepsilon >0$. Let us consider two cases. Assume first that $\mu\left( {\varepsilon}^{-1} Q(I)\right)\leq \varepsilon |I|.$
Then \eqref{final1} gives that 
\[
\fint_{I} |f (\xi) -  f(z_I))| dm (\xi) \lesssim  \varepsilon^{1/2} 
\]
and \eqref{final} follows. Assume now that $\mu\left( {\varepsilon}^{-1} Q(I)\right) >  \varepsilon |I|$. 
Since $|E|\mu$ is a vanishing Carleson measure, Lemma \ref{Lemma 2.4} applied to the family $\mathcal{A}({\varepsilon}^2)$ gives that $|E(z(\varepsilon, I))| \leq \varepsilon, $
if $|I|$ is sufficiently small. Here $z(\varepsilon, I)$ denotes $ z_{{\varepsilon}^{-1} I}$. Since $E \in \VMOA$ we have $(1-|z|) |E' (z)| \to 0$ as $ |z| \to 1$. We deduce that $|E(z_I)|< 2 \varepsilon$ if $|I|$ is sufficiently small. This proves \eqref{final} and finishes the proof

\end{proof}

As usual $L^p(\partial \mathbb D)$ denotes the classical Lebesgue spaces on the unit circle, $0<p\leq \infty$.

\begin{corollary}
\label{cor}
    Let $f\in L^p(\partial \mathbb D)$, $0<p\leq\infty$. Then, there exists an outer function $E \in \QA$ such that $Ef \in \VMO \cap L^\infty(\partial \mathbb D)$.
\end{corollary}

\begin{proof}
Consider the outer function $E_0$ defined as
\[
E_0(z)=\exp\left( -\int_{\partial\mathbb D} \frac{\xi+z}{\xi-z}\log^+|f(\xi)|d\xi\right), \quad z \in \mathbb D.
\]
It is clear that $E_0$ belongs to $\mathbb{H}^\infty$ and $E_0f \in L^\infty(\partial \mathbb D)$. We apply Theorem \ref{Wolff theorem} twice. First, we find an outer function $E_1 \in \QA $ such that  $E_1E_0 \in \QA$. Since $E_1E_0f \in L^\infty(\partial \mathbb D)$, another application of Theorem \ref{Wolff theorem} provides a function $E_2 \in \QA$ such that  $E_2E_1E_0f \in \QA$ and one can take $E= E_2E_1E_0$.
\end{proof}

\medskip 
\subsection{Critical points of functions in Hardy spaces}
Theorem \ref{Theorem 3} follows from a convenient variant of a classical result by W. Cohn on factorization of derivatives of functions in Hardy spaces. Fix $0<p<\infty$. W. Cohn proved in \cite[Theorem 1]{Cohn} that, given $F \in \mathbb{H}^p$, there exist a function $G \in \BMOA$ and an outer function $H \in \mathbb{H}^p$ such that $F' = G' H$. Conversely, for any $G \in \BMOA$ and $H \in \mathbb{H}^p$, the function $G'H$ is the derivative of a function in $\mathbb{H}^p$. See \cite{Dyakonov2012} for a version in the Nevanlinna class. 
Next we apply Theorem \ref{Theorem 1}, Lemma \ref{nou} and a nice technique of \cite{Kraus2013} to show the following result which obviously implies Theorem \ref{Theorem 3}.

\begin{Lemm}\label{factorization2}
Fix $0< p<\infty$. For every $F \in \mathbb{H}^p$ there exist $G\in \QA$ and an outer function $H \in \mathbb{H}^q $ for any $q<p$ such that $F'= G' H$.
\end{Lemm}
\begin{proof}
Let $F \in \mathbb{H}^p$. According to Cohn's result, there exist $\Phi \in \BMOA$ and an outer function $R \in \mathbb{H}^p$ such that $F' = \Phi' R$. Applying Theorem \ref{Theorem 1} to the Carleson measure $|\Phi' (z)|^2 (1- |z|^2) dA(z)$, one obtains an outer function $E \in \QA$ with $\log|E| \in \VMO$, such that 
$E^{1/2}\Phi'$ is the derivative of a function $G \in \VMOA$. Consequently 
\[
F'=G' \frac{R}{E^{1/2}}.
\]
Since $\log|E| \in \VMO$, the John-Nirenberg Theorem gives that $E^{-1/2} \in \mathbb{H}^r$ for every $0<r<\infty$. Holder's inequality gives that $R E^{-1/2} \in \mathbb{H}^q$ for any $0<q<p$. 

Note that the function $G \in \VMOA$ may be unbounded. Next we will apply the technique in \cite{Kraus2013}. Consider the partial differential equation  
\begin{equation}
    \label{PDE}
    \Delta u (z) = 4 |G' (z)|^2 e^{2u(z)} , \quad z \in \mathbb{D}. 
\end{equation}
 Since $G$ is in BMOA the PDE \eqref{PDE} has a solution $u_0$ which is bounded on $\mathbb{D}$ (see Remark 3.4 of \cite{Kraus2013}). By Liouville's Theorem (see Theorem 3.3 of \cite{krausRoth2008}), there exists an analytic self-mapping $I$ of $\mathbb{D}$ such that
$$
u_0 (z)= \log \big( \frac{1}{|G' (z)|} \frac{|I' (z)|}{(1-|I(z)|^2) } \big), \quad z \in \mathbb{D}. 
$$
Since $u_0$ is bounded in $\mathbb{D}$, we deduce that $|G'|$ is comparable to $|I'| / (1- |I|^2)$ on $\mathbb{D}$. Hence $I$ and $G$ have the same critical points with the same multiplicities. Since $G$ is in VMOA, we deduce that $|I'(z)|^2 (1-|z|) dA(z) / (1-|I(z)|)^2 $ is a vanishing Carleson measure. In particular $I$ is in VMOA and hence $I \in \QA$. Note that $|G'  / I'|$ is comparable to $(1-|I|^2)^{-1} \geq 1$ on $\mathbb{D}$ and deduce that $G' / I'$ is an outer function. Finally part (b) of Lemma \ref{nou} gives that $G' / I'$ belongs to the Hardy space $\mathbb{H}^s$, for any $0<s<\infty$.  

\end{proof}

We notice that the proof of the previous Lemma gives the following factorization for the derivatives of $\BMOA$ functions.

\begin{corollary}\label{factorization BMOA}
For every $F \in \BMOA$ there exist $G\in \QA$ and an function $H \in \mathbb{H}^q $ for any $q< \infty$ with $1/H \in \mathbb{H}^\infty$ such that $F' = G' H$. 
\end{corollary}

We close this section with two remarks. A. Aleksandrov and V. Peller proved in \cite[Theorem 3.4]{Alek1996} that for any $F \in \BMOA$ there exist $G_i , H_i \in \mathbb{H}^\infty$, $i=1,2$, such that $F' = G_1' H_1 + G_2'H_2$. The second remark concerns single generated ideals in the space $A^2_1$ of analytic functions $F$ in $\mathbb{D}$ such that
\[
\|F\|^2=\int_{\mathbb D} |F(z)|^2(1-|z|^2)dA(z)<\infty .
\]
O. Ivrii showed that any single generated invariant subspace of $A^2_1$ can be generated by the derivative of a bounded function (see Theorem 3.1 of \cite{ivrii}). 

\begin{corollary}\label{invariant}
Let $F \in A^2_1$ and let $[F]$ denote the closure in $A^2_1$ of polynomial multiples of $F$. Then, there exists a function $G \in \QA$ such that 
$
[G']=[F]
$.
\end{corollary}
\begin{proof}
It is well known that there exists $I \in \BMOA$ such that $[F]=[I']$ (see \cite[Theorem 3.3]{Bergman}). Corollary \ref{factorization BMOA} provides $G \in \QA$ and a function $H \in \mathbb{H}^2$ with $1/H \in \mathbb{H}^\infty$ such that $I'= G' H$. We now verify that  $[I']=[G']$.
Let $W \in [I']$. Given $\varepsilon >0$ there exists an analytic polynomial $P$ such that $\|W-P I' \|\leq \varepsilon.$ Consequently, if $H_n$ is the Taylor polynomial of $H$ of degree $n$, we have 
\begin{align*}
    \|W-P G'H_n\|&\leq \|W-PHG'\|+\|PHG'- PG'H_n\|\\
    &\leq \|W -P I'\|+ C \|G\|_{\BMO}^2 \|P\|_{\infty}\|H-H_n\|_{2},
\end{align*}
where $C>0$ is an absolute constant. The last estimate follows from the fact that $|G' (z)|^2 (1- |z|^2) dA(z)$ is a Carleson measure. Therefore $W \in [G']$. A similar argument using that $1/H \in \mathbb{H}^2$, proves the converse inclusion. 
\end{proof}

\medskip
\subsection{Generalized Volterra operators}
Given two analytic functions $F,G$ in $\mathbb D$, $T_G(F)$ denotes the generalized Volterra operator with symbol $G$ applied to $F$ defined as
\[
T_G(F)(z)=\int_0^zF(w)G'(w)dw, \quad z\in \mathbb{D}.
\]
It is clear that if $G \in \BMOA$, the operator $T_G  \colon \mathbb{H}^\infty \to \BMOA$ is continuous and $\|T_G \| \lesssim \|G \|_{\BMO}$. As it is expected, $T_G$ is compact precisely when the symbol $G \in \VMOA$. 

\begin{Lemm}\label{P:1}
Let $G \in \BMOA$. Then $T_G : \mathbb{H}^\infty \to \BMOA$ is compact if and only if $G \in \VMOA$.   
\end{Lemm}
\begin{proof}
Assume first that $G$ is a polynomial. Note that $T_g=V\circ M_{G'}$, where $M_{G'}\colon \mathbb{H}^{\infty}\to \mathbb{H}^{\infty}$ and $V\colon \mathbb{H}^{\infty}\to \BMOA$ are respectively the operator of multiplication by $G'$ and the classical Volterra operator. By \cite[Theorem 3.5]{Anderson2014}, 
 $V$ acts compactly on $\mathbb{H}^{\infty}$. Hence $T_G$ is compact.  

Consider now an arbitrarily function $G \in \text{VMOA}$. Note that there exist polynomials $P_n$, such that
$$
\lim_{n \to \infty }\|G-P_n\|_{\BMO}=0\, .
$$
Thus 
\begin{align*}
   \lim_{n \to \infty} \|T_G-T_{P_n}\| \lesssim \lim_{n \to \infty}\|G-P_n\|_{\BMO}=0\, .
\end{align*}
Hence, $T_G\colon \mathbb{H}^{\infty}\to \BMOA$ is compact.

The converse is proved by contradiction. Assume that $T_G$ is compact and that $G$ is not in VMOA. Then there exist a constant $M >0$ and a sequence of arcs $\{I_n \}$ in $ \partial \mathbb{D}$ such that $|I_n|\to 0$ and 
\begin{equation}\label{Eq contradicts}
     \frac{1}{|I_n|}\int_{Q(I_n)}|G'(z)|^2(1-|z|^2)dA(z)>M, \quad  n=1,2,\ldots .
\end{equation}
For $n=1,2,\ldots$ pick the integer $N_n$ with $|I_n|^{-1} \leq N_n < |I_n|^{-1} + 1$. Then 
\[
\| T_G (z^{N_n}) \|_{\BMO}^2 \geq \frac{1}{|I_n|}\int_{Q(I_n)}  |z|^{2N_n}  |G'(z)|^2(1-|z|^2)dA(z) \geq M/4. 
\]
Hence $T_G : \mathbb{H}^\infty \rightarrow \BMOA$ is not compact. 
\end{proof}

The boundedness from below of the operators $T_G$ acting on Hardy and Bergman spaces has already  been studied in \cite{Anderson2011} and \cite{Panteris2022}.
The core of our approach lies in establishing that for $G \in \BMOA$, there exists a non zero function $F \in \VMOA \cap T_G(\mathbb{H}^\infty)$.

\begin{proof}[Proof of Theorem \ref{Theorem 4}]
Let $H \in \mathbb{H}^\infty$. Because of Theorem \ref{Theorem 1}, we can find an outer function $E \in \QA$ such that $|E(z) G'(z)H(z)|^2(1-|z|^2)dA(z)$ is a vanishing Carleson measure. This implies that $F = T_G (EH) \in \VMOA$. Note that 
\begin{equation}
    \label{formula}
    E H G' = F' . 
\end{equation}
We argue by contradiction. Assume that $T_G: \mathbb{H}^\infty \to \BMOA$ is bounded from below or that $T_G (\mathbb{H}^\infty )$ is closed in $BMOA$. Then, since $T_G$ is also bounded, there exists a constant $C>0$ such that 
\begin{equation}
C^{-1} \|k\|_{\infty} \leq \|T_G (k)\|_{\BMO} \leq C \|k\|_{\infty}    
\end{equation}
for every $k \in \mathbb{H}^\infty$. In particular, if $k_n(z)=E(z)H(z)z^n$, $z \in \mathbb{D}$, applying \eqref{formula} we deduce that
\begin{align*}
    \|T_G(k_n)\|_{\BMO}^2&\sim \sup_{I}\frac{1}{|I|}\int_{Q(I)}|G'(z)E(z)H(z)|^2|z|^{2n}(1-|z|^2)\, dA(z)\\
    &=\sup_{I}\frac{1}{|I|}\int_{Q(I)}|F'(z)|^2|z|^{2n}(1-|z|^2)\, dA(z) . 
\end{align*}
Since $F \in \VMOA$ we deduce that $\|T_G (k_n)\|_{\BMO} \to 0$ as $n \to \infty$ while $\|k_n \|_\infty$ is bounded below.

\end{proof}    
\medskip

\section{Sharpness of Theorem \ref{Theorem 1}}
We first show that in the conclusion of Theorem \ref{Theorem 1}, one can not improve the vanishing Carleson measure condition.

\begin{prop}\label{Prop:optimality vanishing carleson}
 For any increasing function $\omega\colon [0,1]\to [0,+\infty)$ with $\omega(0)=0$, there exists a Carleson measure $\mu$ such that for any outer function $E \in \mathbb{H}^\infty$ we have 
 \[
 \limsup_{\ell(Q)\to 0}\frac{\int_Q|E(z)|d\mu(z)}{\ell(Q)\omega(\ell(Q))}=+\infty.
 \]
\end{prop}
\begin{proof}
Pick two sequences $\{\delta_k\}$ and $\{h_k\}$ of positive numbers such that
\begin{equation}\label{E:conditiondeltah}
  \sum_{k=1}^\infty \frac{h_k}{\delta_k}<\infty  \quad \text{ and }\quad  \lim_{k \to \infty}\frac{h_k}{\delta_k} \log(\omega(h_k))=-\infty.
\end{equation}
We can also assume that $N_k = 2 \pi / \delta_k$ is an integer for any $k=1,2,\ldots$. For any $k=1,2,\ldots$, let $\Lambda_k =\{ z_{k,j} : j=1,\ldots , N_k \}$ be points uniformly distributed in the circle $\{z \in \mathbb{D} : 1-|z| = h_k\}$. 
Then the sequence $\{z_n\}$ defined as 
\[
\{z_n\}=\bigcup_{k=1}^\infty \Lambda_k
\]
is a Blaschke sequence.
Actually 
\[
\mu=\sum_{n}(1-|z_n|)\delta_{z_n},
\]
is a Carleson measure. 
We argue by contradiction and assume that there exists an outer function $E \in \mathbb{H}^\infty$, $\|E\|_\infty \leq 1$, such that
\begin{equation}\label{eqA}
    \frac{\int_Q|E(z)|d\mu(z)}{\ell(Q)\omega(\ell(Q))}<1
\end{equation}
for every Carleson square $Q$. We pick Carleson squares $Q_{k,j}$ with $\ell(Q_{k.j})=h_k$ such that $z_{k,j}$ lies in the top part $T(Q_{k.j})$ of $Q_{k.j}$. 
The assumption \eqref{eqA} implies that for every $k,j$
\begin{align*}
|E(z_{k,j})|h_k\leq & 
\ell(Q_{k,j})\omega(\ell(Q_{k,j}))=h_k\omega(h_k),
\end{align*}
that is, 
\begin{equation}\label{eqB}
|E(z_{k,j})|\leq \omega(h_k), \quad j=1, \ldots , N_k ; \, k=1,2,\ldots .
\end{equation}
Consider the discs $D_{k,j} = \{z \in \mathbb{D} : |z - z_{k,j}| \leq (1-|z_{k,j}|) / 2 \}$. Harnack's inequality applied to the positive harmonic function $- \log |E|$ gives that
\[
-\log|E(z)| \gtrsim -\log |E(z_{k,j})| \geq -\log\omega(h_k), \quad z \in D_{k,j}, \quad j=1, \ldots , N_k ; \,  k=1,2,\ldots .
\]
uniformly in $k,j$.
Consequently, by subharmonicity, we have that
\begin{align*}
\log|E(0)|&\leq \int_{\partial \mathbb{D}}\log|E((1- h_k) \xi)|dm (\xi)
\lesssim \log(\omega(h_k)) \frac{h_k}{\delta_k} \to -\infty , \, \text{as} \, k \to \infty.
\end{align*}
This is clearly a contradiction and finishes the proof. 
\end{proof}


We recall that the disc algebra $\A$ is the space of continuous functions in the closed unit disc which are analytic in $\mathbb{D}$. Our next result says that in Theorem \ref{Theorem 1} one  can not replace the condition $E \in \QA$ by $E \in \A$.
\begin{prop}\label{Optimality not E in A}
There exists a Carleson measure  $\mu$ in $\mathbb D$ such that there are no non trivial functions $E \in \A$ such that $|E|\mu$ is a vanishing Carleson measure. 
\end{prop}
\begin{proof}
Consider a Carleson measure $\mu$ such that for any point  $\xi \in \partial  \mathbb D $ there exists a sequence of Carleson squares  $\{Q_n(\xi)\}$ such that $\mu(Q_n(\xi)) > \ell(Q_n(\xi))$ and  
\[
\lim_{n \to \infty}\ell(Q_n(\xi))=0\, , \lim_{n \to \infty}\text{dist}(\xi, Q_n(\xi))=0.
\]
For instance one could consider a uniformly separated sequence $\Lambda$ with $\partial \mathbb{D} \subset \overline{\Lambda}$ and 
$$
\mu = \sum_{z \in \Lambda} (1-|z|) \delta_z . 
$$
Assume that there exists a function $E \in \A$ such that $|E|\mu$ is a vanishing Carleson measure. Then
\[
\lim_{n \to \infty}\frac{\int_{Q_n(\xi)}|E|d\mu}{\ell(Q_n(\xi))}=0, \quad \xi \in \partial \mathbb{D} . 
\]
Since $E$ is continuous in $\overline{\mathbb{D}}$, this implies that $E$ vanishes identically. 
\end{proof}

Our last remark concerns the sharpness of part (b) of Theorem \ref{Theorem 1}: one can not replace the condition $\log |E| \in \BMO$ by the stronger one $\log |E| \in \VMO$. 
\begin{prop}\label{Missing optimal}
There exists a finite positive Borel measure $\mu$ in $\mathbb{D}$ such that there are no function $E \in \QA$ with $\log|E|\in \VMO$ such that $|E| \mu$ is a vanishing Carleson measure.
\end{prop}
\begin{proof}
We argue by contradiction. Let $G \in \mathbb{H}^2$ and consider the measure $\mu(z) = |G'(z)|^2 (1- |z|^2) dA(z)$. Assume there exists an outer function $E \in \QA$ such that $|E| \mu$ is a vanishing Carleson measure and $\log|E| \in \VMO$. Then the function $F$ defined by
\[
F(z) = \int_0^z E(w) G' (w) dw, \quad z \in \mathbb{D}
 \]
belongs to $\VMOA$ and satisfies $F' = G' E$. 
Since $\log |E| \in VMO$, we have $1/E \in \mathbb{H}^p$ for $p>2$. hence we obtain that for any function $G \in \mathbb{H}^2$ one can factor $G'= F' / E \in \{H' : H \in \mathbb{H}^p\} $ by the result of W.Cohn \cite{Cohn}, 
which is clearly a contradiction if $p>2$.
\end{proof}

As explained in the introduction, the proof of Theorem \ref{Theorem 1} relies on the fact that for any finite positive Borel measure $\mu$, the ratio $\mu (Q) / \ell (Q)$ is small for {\it most} Carleson squares $Q \subset \mathbb{D}$. Our last result points in this direction.

\begin{prop}\label{L:prel}
Let $\mu$ be a finite positive measure. Then 
\[
\lim_{h\to 0 }\frac{\mu\left( Q(\xi,h)\right)}{h}=0\, ,
\]
for $m$-almost every $\xi \in \partial \mathbb D$. Here $Q(\xi,h) = Q(I(\xi, h))$ where $I(\xi, h)$ is the arc of $\partial \mathbb{D}$ centered at $\xi$ of normalized length $h$.
\end{prop}
\begin{proof}
We proceed by contradiction, that is, we assume that  
\begin{equation}
    m\left\lbrace \xi \in \partial \mathbb{D} : \limsup_{h\to 0 }\frac{\mu\left( Q(\xi,h)\right)}{h}>0\right\rbrace > 0\, .
\end{equation}
By the regularity of the Lebesgue measure, there exists a constant $\eta >0$ such that
\begin{equation*}
    m \left\lbrace \xi \in \partial \mathbb{D} : \limsup_{h\to 0 }\frac{\mu\left( Q(\xi,h)\right)}{h}>\eta\right\rbrace > 0\, .
\end{equation*}
Let $\varepsilon >0$ be a small number to be fixed later. Consider a compact set $K\subset \partial \mathbb D$  wih $m(K) >0$ such that for every $\xi \in K$ there exists a sequence $\{h_n(\xi)\}_n$ tending to $0$ as $n\to \infty$ with $0< h_n (x) < \varepsilon$ and 
$$
\mu(Q(\xi, h_n(\xi)))>\dfrac{\eta}{2}\ h_n(\xi), \quad n=1,2, \ldots .
$$ 
We consider the collection of arcs $\{I(\xi, h_n (\xi)) : n=1,2,\ldots ; \xi \in K\}$. Using Vitali's covering lemma, we extract a family of pairwise disjoint arcs $\{I_k \}$ such that $K\subseteq \bigcup_k 5 I_k $.
Note that $\mu (Q(I_k)) \geq \eta |I_k| / 2$. Consequently
\begin{align*}
   \mu\left(\left\lbrace z \in \mathbb D : 1-\varepsilon\leq |z|<1\right\rbrace\right)&\geq \mu\left(\cup_k Q(I_k) \right)
   \geq \frac{\eta}{2} \sum_k |I_k| \geq \frac{\eta}{10} |K|\, .
\end{align*}
However, the last inequality gives a contradiction if $\varepsilon >0$ is taken sufficiently small since 
$$
\lim_{\varepsilon \to 0}\mu\left(\left\lbrace z \in \mathbb D : 1-\varepsilon\leq |z|<1\right\rbrace\right)=0\, .
$$
\end{proof}




\medskip

\section{Acknowledgements}
The first author is member of Gruppo Nazionale per l’Analisi Matematica, la Probabilit\`a e le loro Applicazioni (GNAMPA) of Istituto Nazionale di Alta Matematica (INdAM) and he was supported by PID2021-123405NB-I00 by
the Ministerio de Ciencia e Innovaci\'on. 

The second author is supported in part by the Generalitat de Catalunya (grant 2021 SGR 00071), the Spanish Ministerio de Ciencia e Innovaci\'on (project  PID2021-123151NB-I00) and the Spanish Research Agency (Mar\'ia de
Maeztu Program CEX2020-001084-M ).

The third author was partially supported by the Hellenic Foundation for Research and Innovation (H.F.R.I.) under the '2nd Call for H.F.R.I. Research Projects to support Faculty Members \& Researchers' (Project Number: 4662).

\bibliographystyle{alpha}
\bibliography{sample}
\medskip 
\end{document}